\documentclass[11pt]{article}
\usepackage{amsfonts}
\usepackage{latexsym}
\usepackage{cite}
\usepackage{amsmath,amsfonts,latexsym,amssymb}
\usepackage[mathscr]{eucal}
\usepackage{cases}
\usepackage{amsthm}

\usepackage[bf,small]{caption2}
\usepackage{float}
\usepackage{graphicx}
\usepackage{amsmath}
\usepackage{amssymb}
\usepackage[all]{xy}

\newtheorem{theorem}{theorem}[section]
\newtheorem{thm}[theorem]{Theorem}
\newtheorem{lem}[theorem]{Lemma}

\newtheorem{exmp}[theorem]{Example}

\newtheorem{rmk}[theorem]{Remark}

\begin{document}

\title{\textbf{Counting homotopy classes of mappings via Dijkgraaf-Witten invariants}}
\author{\Large Haimiao Chen
\footnote{Email: \emph{chenhm@math.pku.edu.cn}}\\
\normalsize \em{Beijing Technology and Business University, Beijing, China}}
\date{}
\maketitle

\begin{abstract}
Suppose $\Gamma$ is a finite group acting freely on $S^{n}$ ($n\geqslant 3$ being odd) and $M$ is any closed oriented $n$-manifold. We show that,
given an integer $k$, the set $\deg^{-1}(k)$ of based homotopy classes of mappings with degree $k$ is finite and its cardinality depends only on the congruence class of $k$ modulo $\#\Gamma$; moreover, $\#\deg^{-1}(k)$ can be expressed in terms of the Dijkgraaf-Witten invariants of $M$.
\end{abstract}

\textbf{key words:} homotopy class, degree, topological spherical space form, Dijkgraaf-Witten invariant.


\section{Introduction}


The topic of nonzero degree mappings between manifolds has a history of longer than 20 years. For a 3-manifold $M$, the homotopy set $[M,\mathbb{RP}^{3}]$ is important in the study of Lorentz metrics on a space-time model in which $M$ corresponds to the space part (c.~f.~\cite{lens}), and in this context the existence of a degree-one mapping of any quotient of $S^{3}$ by a free action of a finite group to $\mathbb{RP}^{3}$ was completely determined by Shastri-Zvengrowski \cite{types}. The existence of degree-one mappings from 3-manifolds $M$ to lens spaces was studied by Legrand-Wang-Zieschang \cite{lens} and some conditions were obtained and expressed in terms of the torsion part of $H_{1}(M)$. Based on their work of computing cohomology ring, Bryden-Zvengrowski \cite{Seifert} gave sufficient and necessary conditions for the existence of degree-one mappings from Seifert 3-manifolds to lens spaces in terms of geometric data.
As for another problem, the set of self-mapping degrees of 3-manifolds in Thurston's picture are all determined by Sun-Wang-Wu-Zheng\cite{self}. In high dimension, degrees of mappings between $(n-1)$-connected $2n$-manifolds were studied in Ding-Pan\cite{general2}, Duan-Wang \cite{general} and Lee-Xu \cite{general3}.
For other problems, see \cite{geometric, geometric2} and the references therein.

\medskip

In this article we focus on mappings from a closed oriented $n$-manifold $M$ (with $n\geqslant 3$ odd) to a \emph{topological spherical space form} $S^{n}/\Gamma$ where $\Gamma$ is a finite group with an orientatation-preserving free action on $S^{n}$.
Nowadays we have a fairly good understanding of such $\Gamma$ in general and a complete classification of the spaces $S^{n}/\Gamma$ when $n=3$. An important achievement of geometric topology is the characterization of finite group $\Gamma$ which can act freely on some $S^{n}$ by Madsen-Thomas-Wall \cite{MTW}: a finite group $\Gamma$ can act freely on some $S^{n}$ if and only if $\Gamma$ satisfies the $p^{2}$- and $2p$-conditions for all prime $p$, i.~e.~all subgroups of $\Gamma$ of order $p^{2}$ or $2p$ are cyclic. In dimension $3$, a more concrete classification is known: by the work of Perelman, all $3$-dimensional topological spherical space forms are spherical space forms, i.~e.~$\Gamma < SO(4)$ acting on $S^{3}$ by isometries; and the determination of $3$-dimensional spherical space forms is a classical result due to Threlfall-Seifert and Hopf in the 1930's (\cite{Hopf, ThS}). Besides the cyclic group case (where $S^{3}/\Gamma$ is a lens space), there are four cases: the dihedral case, the tetrahedral case, the octahedral case and the icosahedral case (where the Poincar\' e $3$-sphere is an example). For more details see \cite{Volk1, Volk2}.

It is a classical result of \cite{degree} that the homotopy classes of mappings $f:M\rightarrow S^{n}/\Gamma$ are classified by the degree $\deg f$ and the induced homomorphism $\pi_{1}(f):\pi_{1}(M)\rightarrow\Gamma$. We find that the based homotopy set $[M,S^{n}/\Gamma]$ has a nice structure, namely, it is an affine set modeled on $\mathbb{Z}$. Given an integer $k$, the subset of homotopy classes of mappings with degree $k$ is a finite set and its cardinality can be expressed in terms of the Dijkgraaf-Witten invariants of $M$.

Dijkgraaf-Witten theory was first proposed in \cite{DW} by the two authors naming the theory, as a 3-dimensional
topological quantum field theory. Later it was generalized to any dimension by Freed
\cite{Freed}. Recall that for a given finite group $\Gamma$ and a cohomology class $[\omega]\in H^{n}(B\Gamma;U(1))$, the \emph{Dijkgraaf-Witten invariant} of a closed $n$-manifold $M$ is defined as
$$
Z^{[\omega]}(M)=\frac{1}{\#\Gamma}\cdot\sum\limits_{\phi:\pi_{1}(M)\rightarrow\Gamma}\langle [\omega],f(\phi)_{\ast}[M]\rangle,
$$
where $f(\phi)_{\ast}:H_{n}(M;\mathbb{Z})\rightarrow H_{n}(B\Gamma;\mathbb{Z})$ is induced by the associated mapping $f(\phi):M\rightarrow B\Gamma$ which is unique up to based homotopy. The TQFT axioms enable us to compute $Z^{[\omega]}(M)$ by a cut-and-paste process.

\section{Mappings to topological spherical space forms}

We assume that all manifolds are closed, oriented and equipped with a base-point, and all mappings and homotopies are base-point preserving.

From now on we suppose $\Gamma$ acts freely on $S^{n}$ and denote $m=\# \Gamma$.
The classifying space $B\Gamma$ can be obtained by attaching higher dimensional cells to $S^{n}/\Gamma$ to kill homotopy groups.
Let $\varphi:S^{n}/\Gamma\rightarrow B\Gamma$ denote the inclusion map, and let $p:S^{n}\rightarrow S^{n}/\Gamma$ denote the covering map.

\begin{lem} \label{lem}
Suppose $M$ is an $n$-manifold, and $f_{0},f_{1} \colon M\rightarrow S^{n}/\Gamma$ are two mappings such that
$\pi_{1}(f_{0})=\pi_{1}(f_{1}):\pi_{1}(M)\rightarrow\Gamma$, then $\deg f_{0}\equiv\deg f_{1}\pmod{m}$. If furthermore $\deg f_{0}=\deg f_{1}$, then $f_{0}$ is homotopic to $f_{1}$.
\end{lem}

\begin{proof}
The composition $\varphi \circ f_{0}$ is homotopic to $\varphi \circ f_{1}$ since $\pi_{1}(f_{0})=\pi_{1}(f_{1}):\pi_{1}(M)\rightarrow\Gamma$; choose a homotopy
$h:M\times[0,1]\rightarrow B\Gamma$.
Let $K \subset M$ be the $(n-1)$-skeletion of $M$, then $M=K \cup D^{n}$. By cellular approximation, we may find $h':K\times[0,1]\rightarrow S^{n}/\Gamma$ such that $\varphi\circ h'\simeq h|_{K\times[0,1]}:K\times[0,1]\rightarrow B\Gamma$, and $h'|_{K\times\{i\}}=f_{i}$ for $i=0,1$.

From $f_{0},f_{1}$ and $h'$ we can construct a mapping
$$g \colon S^{n}=\partial (D^{n} \times [0,1]) \hookrightarrow M\times\{0,1\}\cup K \times [0,1] \to S^{n}/\Gamma.$$
Let $g\cdot f_{0}$ denote the composite
$$M\cong S^{n}\# M\rightarrow S^{n}\vee M\stackrel{\mathrm{g\vee f_{0}}}{\longrightarrow}S^{n}/\Gamma.$$
Then $g\cdot f_{0}\simeq f_{1}$, hence $\deg g=\deg f_{1} - \deg f_{0}$.

Clearly $g$ lifts to a mapping $\overline{g} \colon S^{n} \to S^{n}$ and $\deg g=m\deg \overline{g}$.
\end{proof}

\medskip

From the proof we see that there is a free action of $[S^{n},S^{n}/\Gamma]\cong\mathbb{Z}$ on $[M,S^{n}/\Gamma]$. By cellular approximation, each mapping $M\rightarrow B\Gamma$ is homotopic to a mapping $M\rightarrow S^{n}/\Gamma$. Summarizing, we have a diagram
$$\xymatrix{
&[M,S^{n}/\Gamma]\ar[d]^{\deg}\ar[r]^{\pi_{1}}&\hom(\pi_{1}(M),\Gamma) \\
&\mathbb{Z}}
$$
such that
\begin{itemize}
\item The map $\pi_{1}$ is surjective, and $\deg\times\pi_{1}:[M,S^{n}/\Gamma]\rightarrow\mathbb{Z}\times\hom(\pi_{1}(M),\Gamma)$ is injective.
\item For any $\phi\in\hom(\pi_{1}(M),\Gamma)$, the set $\deg(\pi_{1}^{-1}(\phi))$ is a congruence class modulo $m$.
\end{itemize}

It is well-known that (see Section 1.6 of \cite{cohomology}, for instance), by splicing the acyclic $\mathbb Z \Gamma$ complex
$$0 \to \mathbb Z \to C_{n}(S^{n}) \to \cdots \to C_{0}(S^{n}) \to \mathbb Z \to 0$$
we obtain a periodic resolution of $\mathbb Z$ as a trivial $
\mathbb Z \Gamma$-module and hence $H_{n}(B\Gamma ; \mathbb Z)\cong\mathbb Z/m\mathbb Z$ with a preferred generator coming from the fundamental class of $S^{n} / \Gamma$. By the Universal Coefficient Theorem,
$$
H^{n}(B\Gamma;U(1))\cong\hom(H_{n}(B\Gamma;\mathbb{Z}),U(1)) \cong \mathbb Z /m\mathbb Z$$
and the pairing $H^{n}(B\Gamma;U(1))\times H_{n}(B\Gamma;\mathbb{Z})\rightarrow U(1)$ is given by
\begin{align}
\langle\overline{l},\overline{k}\rangle=\zeta_{m}^{kl},
\end{align}
where
\begin{align}
\zeta_{m}=\exp(\frac{2\pi i}{m}).
\end{align}

\begin{thm} \label{thm}
Each homotopy set $\deg^{-1}(k)$ is finite.
More precisely,
$$
\#\deg^{-1}(k)=\sum\limits_{\overline{l}\in\mathbb{Z}/m\mathbb{Z}}Z^{\overline{l}}(M)\cdot\zeta_{m}^{-kl}.
$$
\end{thm}
\begin{proof}
There is a one-to-one correspondence
$$\deg^{-1}(k)\leftrightarrow\{\phi:\pi_{1}(M)\rightarrow\Gamma\colon f(\phi)_{\ast}[M]=\overline{k}\},$$
hence
\begin{align*}
\#\deg^{-1}(k)&= \#\{\phi:\pi_{1}(M)\rightarrow\Gamma\colon f(\phi)_{\ast}[M]=\overline{k}\}, \\
Z^{\overline{l}}(M)&= \frac{1}{m}\cdot\sum\limits_{\overline{k}\in\mathbb{Z}/m\mathbb{Z}}\#\deg^{-1}(k)\cdot\zeta_{m}^{kl}.
\end{align*}
The formula for $\#\deg^{-1}(k)$ is obtained by Fourier transformation.
\end{proof}

\begin{rmk}
\rm In \cite{linking} Murakami-Ohtsuki-Okada defined an invariant of 3-manifold which can be shown to depend only on $\beta_{1}(M)$ (the first Betti number of $M$) and the linking paring $\lambda$ on $\textrm{Tor}H_{1}(M)$. They proved a formula expressing the DW invariant $Z^{\overline{l}}(M)$ (for a cyclic group) in terms of their invariant. Combining this with our result, we see that in dimension $3$ there is an indirect connection from $\beta_{1}(M)$ and $\lambda$ to $\#\deg^{-1}(k)$, and the result of \cite{lens} provided an evidence for the case $k=1$.
\end{rmk}

\begin{exmp}
\rm Let us consider the special case when $n=3$, $M$ is the Seifert 3-manifold with orientable base $M_{O}(g;(a_{1},b_{1}),\cdots,(a_{r},b_{r})))$ and $\Gamma=\mathbb{Z}/m\mathbb{Z}$. The Dijkgraaf-Witten invariants of such manifolds are computed by the first author in \cite{chm}. By Theorem 3.3 of \cite{chm},
\begin{align}
Z^{\overline{l}}(M)=m^{2g-2}\sum\limits_{h,s\in\mathbb{Z}/m\mathbb{Z}}\prod\limits_{j=1}^{r}\left(\sum\limits_{z_{j}
    \in\mathbb{Z}/m\mathbb{Z}\atop a_{j}z_{j}=h}\zeta_{m^{2}}^{\tilde{l}a_{j}b_{j}\tilde{z_{j}}^{2}-(2\tilde{l}\tilde{h}+m\tilde{s})b_{j}\tilde{z_{j}}}\right),
\end{align}
where $x\mapsto\tilde{x}$ denotes the obvious map $\mathbb{Z}/m\mathbb{Z}\rightarrow\{0,1,\cdots,m-1\}$.

Thus
\begin{align}
&\#\deg^{-1}(k) \nonumber \\
=&m^{2g-2}\sum\limits_{l,h,s\in\mathbb{Z}/m\mathbb{Z}}\zeta_{m}^{-kl}\prod\limits_{j=1}^{r}\left(\sum\limits_{z_{j}
\in\mathbb{Z}/m\mathbb{Z}\atop a_{j}z_{j}=h}\zeta_{m^{2}}^{\tilde{l}a_{j}b_{j}\tilde{z_{j}}^{2}-(2\tilde{l}\tilde{h}+m\tilde{s})b_{j}\tilde{z_{j}}}\right).
\end{align}
\end{exmp}

\medskip

{\bf Acknowledgements}\\
I would like to express my special gratitude to Professor Yang Su at Institute of Mathematics, Chinese Academic of Science, for many beneficial conversations.

\end{document}